\documentclass{amsart}      
\usepackage{amssymb,amsmath,graphicx}
\newtoks\prt 
\numberwithin{equation}{section}
 
\newtheorem{thm}{Theorem}
 
\newtheorem{lemma}[thm]{Lemma}

\theoremstyle{definition} 
 
\newtheorem{remark}[thm]{Remark}

\def\eqn#1$$#2$${\begin{equation}\label#1#2\end{equation}}

\def\en{\mathbb N} 
\def\er{\mathbb R}

\def \Int {\operatorname{Int}}

\def\Ker{\operatorname{Ker}}

\def \reg {\partial _{\kern1pt\text{reg}}}

\newcommand{\cs}{\operatorname{cs}}

\newcommand{\cc}{\operatorname{cc}}

\begin{document}

\title{Note on Bessaga-Klee classification}
\author{Marek C\'uth and Ond\v{r}ej F.K. Kalenda}

\address{Department of Mathematical Analysis \\
Faculty of Mathematics and Physic\\ Charles University\\
Sokolovsk\'{a} 83, 186 \ 75\\Praha 8, Czech Republic}
\email{marek.cuth@gmail.com}
\email{kalenda@karlin.mff.cuni.cz}


\thanks{Our research was supported in part by the grant GA\v{C}R P201/12/0290.}

\begin{abstract} 
We collect several variants of the proof of the third case of the Bessaga-Klee relative classification of closed convex bodies in topological vector spaces. We were motivated by the fact that we have not found anywhere in the literature a complete correct proof. In particular, 
we point out an error in the proof given in the book of C.~Bessaga and A.~Pe\l czy\'nski (1975). We further provide a simplified version of T.~Dobrowolski's proof of the smooth classification of smooth convex bodies in Banach spaces which works simultaneously in the topological case.
\end{abstract}
\maketitle


\section{Introduction}

A well-known result due to Bessaga and Klee (see, for example, \cite[Section III.6]{BePe}) provides a classification of pairs $(X,U)$, where $X$ is a Hausdorff topological vector space and $U\subset X$ a closed convex body, up to a homeomorphism. Let us recall this result.

Let $X$ be a Hausdorff topological vector space and $U\subset X$ a closed convex body (i.e., a closed convex set with nonempty interior). The  {\it characteristic cone} of $U$ (denoted by $\cc U$) is the set of those $x\in X$ such that the half-line $a+[0,+\infty)x$ is contained in $U$ for some $a\in U$. If $0\in\Int U$, then $\cc U$ is exactly the zero set of the Minkowski functional of $U$ (see, e.g., \cite[Section III.1]{BePe}).

Then the classification is summed up in the following theorem:

\begin{thm}\label{t:BK}
Let $X$ be a Hausdorff topological vector space and $U\subset X$ a closed convex body.
\begin{itemize}
	\item[(i)] If $\cc U$ is a linear subspace of finite codimension $m$, then the pair $(X,U)$ is homeomorphic to the pair $(\cc U\times \er^m,\cc U\times[0,1]^m)$.
	\item[(ii)] If $\cc U$ is a linear subspace of infinite codimension, then
	the pair $(X,U)$ is homeomorphic to the pair $(X,X^+)$, where $X^+$ is a closed half-space of $X$.
	\item[(iii)] If $\cc U$ is not a linear subspace, then the pair $(X,U)$ is homeomorphic also to the pair $(X,X^+)$.
\end{itemize}
\end{thm}

We have studied this result at a seminar with students using the book \cite{BePe} and we encountered a difficulty with proving the assertion (iii). 
On page 112 of that book a formula is given, illustrated by a picture and followed by the claim that `it is not difficult to check' that this gives the required homeomorphism. After certain effort we realized that this claim is not true -- the given formula need not provide a homeomorphism. It is explained in Section~\ref{s:BP} below.

After finding the error we tried to correct it and to look at the literature for a correct proof. The original reference for the result is the paper \cite{BeKl}. However, this paper does not contain explicit formulation of the theorem. The desired statement is a special case of the more general \cite[Lemma 1.3]{BeKl}.
And again, in its proof a formula is described and followed by the claim that `it is tedious but not difficult to verify' that the formula gives the desired homeomorphism. In this case the claim is correct. In fact, the proof is not even too tedious. In Section~\ref{s:BK} we describe this method applied directly to the case of the above theorem. 

Before finding and analyzing the original paper we established a correction 
of the proof from \cite{BePe}. This correction is described in Section~\ref{s:C1}.
It is quite complicated, but we think it contains several interesting features.
Later, after analyzing the original method we got an idea that the error in \cite{BePe} is probably due to a misprint. And really, this yields the proof described in Section~\ref{s:C2}. The proof is a bit more complicated than the original one.

Finally, we found the paper \cite{Dob} where an analogous classification of $C^p$-smooth convex bodies in Banach spaces up to a $C^p$-diffeomorphism is given. As a special case $p=0$ the homeomorphic classification is given. The proof of the case (iii) takes only half a page. It refers to the implicit function theorem  
\cite[Theorem 10.2.5]{Dd}. However, the key parts of the proof are missing (for example the proof that the respective maps are bijections and the proof that the Fr\'echet differential at each point is an onto isomorphism). Further, there is 
one small mistake in the definition of one of the important sets. 
In Section~\ref{s:D} below we give a proof using the method of \cite{Dob} for
the homeomorphism case. Under the additional smoothness assumptions the same proof
provides the classification up to a diffeomorphism. Further, our proof is more elementary, since it uses only a simple version of the implicit function theorem (see Theorem~\ref{IFT} below).
 
In view of this situation we decided to write down several variants of the proof because we think that such a result deserves it.

\medskip

Let us fix some notation. We adopt the notation of \cite{BePe}, the notation in the other two works is different. 

If $U$ is a convex set containing $0$ in its interior, we denote by $w_U$ the Minkowsi functional of $U$. Further, $\cs U$ is the set of those $x\in U$ such that the line $a+\er x$ is contained in $U$ for some $a\in X$. In other words,
$\cs U=\cc U\cap \cc(-U)$.

\section{The basic method of the proof}

We will review below several possibilities of proving the assertion (iii) of Theorem~\ref{t:BK}. Not surprisingly, all the proofs follow the same pattern.
Let us describe this general pattern.

Let $U\subset X$ be a closed convex body such that $\cc U$ is not a linear subspace. It means that there is $y\in \cc U$ such that $-y\notin\cc U$. Without loss of generality we may suppose that $0\in \Int U$. Then $[0,+\infty)y\subset U$, $(-\infty,0]y\not\subset U$ and there is some $\varepsilon>0$ such that $(-\varepsilon,0]y\subset U$. Hence, without loss of generality we may suppose that $-y\in\partial U$. If we define a linear functional on $\er y$ by the formula $\psi_0(ty)=-t$, then $\psi_0(ty)\le w_U(ty)$ for each $t\in\er$. So, Hahn-Banach theorem implies that there is a linear functional $\psi$ on $X$ extending $\psi_0$ such that $\psi(x)\le w_U(x)$ for each $x\in X$. Set
$\varphi=-\psi$. Then $\varphi$ is a linear functional on $X$ such that
$\varphi(-y)=-1$ and $\varphi(x)\ge -1$ for $x\in U$. In particular, $|\varphi(x)|\le 1$ on $U\cap (-U)$, so $\varphi$ is continuous. 
Set $Z=\{x\in X:\varphi(x)=-1\}$.

Now, a basic method of constructing a homeomorphism of the pair $(X,U)$ onto the pair $(X,\varphi^{-1}([-1,+\infty))$ is the following: To any $z\in Z$ assign some $c(z)\in [-1,+\infty)y$. Let $u(z)$ be the last point at the segment $[c(z),z]$ contained in $U$ and let $v(z)$ be a suitable point at the segment $(c(z),u(z))$.
Next, we choose a self-homeomorphism $h_z$ of the halfline $c(z)+(0,+\infty)(z-c(z))$ which is identity on the segment $(c(z),v(z)]$ and the segment $[v(z),u(z)]$ is mapped onto the segment $[v(z),z]$. Finally define the global homeomorphism $H$ by $h_z$ at the respective halfline and by the identity at the points not covered by such halflines.



\includegraphics{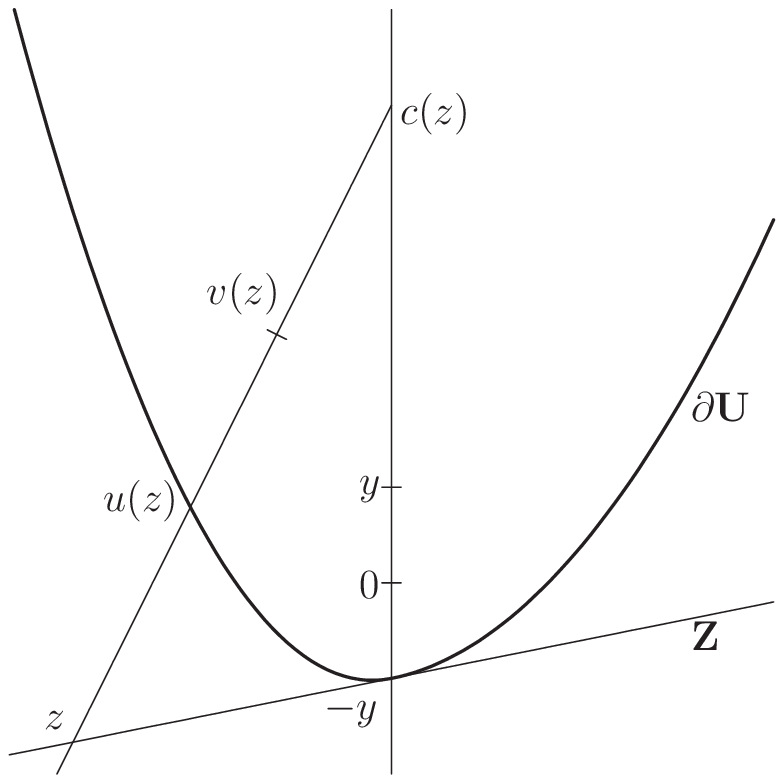}

\vskip60mm 

Then a proof that $H$ is indeed a homeomorphism requires three steps:
\begin{itemize}
	\item $H$ is well-defined (i.e., the respective halflines do not intersect).
	\item $H$ is a self-homeomorphism of the union of the halflines.
	\item $H$ remains homeomorphism if glued with the identity.
\end{itemize}

The proofs appearing in the literature differ in the formula for $c(z)$, the choice of $v(z)$ and the definition of $h_z$.

An important part of the proof (namely of the second step) consists in using the following easy lemma.

\begin{lemma}\label{l:Mink} Let $U\subset X$ be a closed convex body. Then the mapping
$$(u,v)\mapsto w_{U-u}(v)$$
is continuous on $\Int U\times X$.
\end{lemma}

\begin{proof} Let $c\in\er$ be arbitrary. We will show that the sets
$$\{(u,v)\in\Int U\times X:w_{U-u}(v)<c\}\mbox{ and }\{(u,v)\in\Int U\times X:w_{U-u}(v)>c\}$$
are open.

If $c\le 0$, then the first set is empty. For $c>0$ the inequality $w_{U-u}(v)<c$ is equivalent to $v\in c\Int(U-u)$, so $v+cu\in \Int U$. It follows that the first set is in this case open.

The second set equals $\Int U\times X$ for $c<0$. For $c=0$ it equals
$\Int U\times(X\setminus\cc U)$. Finally, for $c> 0$
the inequality $w_{U-u}(v)>c$ is equivalent to $v\notin c(U-u)$, i.e., 
$v+cu\in X\setminus cU$. In any case the second set is open as well.
\end{proof}

\section{Several variants of the proof}

In this section we collect several variants of the proof. We start by the original 
proof which is hidden in \cite{BeKl}, then we continue by explaining why the proof in \cite{BePe} is incorrect and suggest two possible corrections. 

\subsection{The original proof}\label{s:BK}

As we have remarked above, the paper \cite{BeKl} in fact do not contain explicit formulation of the theorem. 
But the result follows from a more general Lemma 1.3. Let us give the proof to see that it is really easy, 
if properly formulated.

Fix a closed convex body $V$ such that $[0,+\infty)y\subset \Int V\subset V\subset \Int U$.
For example, one can take $V=\frac12U$ or $V=\frac y2+U$. Set $W=V\cap(-V)\cap\Ker\varphi$. Then $W$ is a closed convex body in $\Ker\varphi$ and, moreover, $\cc W=\cs W=\cs V$.

For $z\in Z$ define $c(z)=w_W(z+y)y$ and let $v(z)$ be the last point of the segment $[c(z),z]$ contained in $V$. The homeomorphism $h_z$ is defined as identity on the segment $(c(z),v(z)]$, on $[v(z),u(z)]$ as the affine transformation sending this segment to $[v(z),z]$ and on the halfline $u(z)+(0,+\infty)(z-c(z))$ as a translation.

\includegraphics{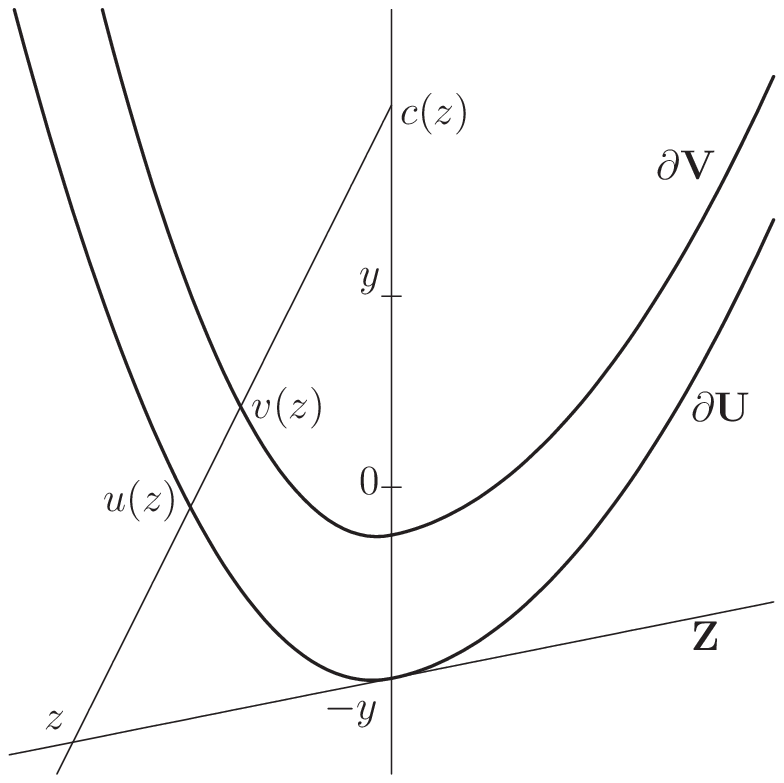}

\vskip60mm 
The proof that the glued mapping $H$ is a homeomorphism has three steps:

\medskip

\noindent{\sc Step 1:} The halflines $c(z)+(0,+\infty)(z-c(z))$, $z\in Z$, are pairwise disjoint and their union is the set $X\setminus (\cs V+[0,+\infty)y)$.

\smallskip

Let $x\in X$. Let us find out under which conditions there is $z\in Z$ such that $$x\in c(z)+(0,+\infty)(z-c(z)),$$ i.e., there are $z\in Z$ and $\alpha>0$ such that
\begin{equation}\label{eq:zero}x=c(z)+\alpha(z-c(z)).\end{equation}
This equation is equivalent to
\begin{equation}\label{eq:first}(x-\varphi(x) y) + \varphi(x) y
= \alpha (z+y) + ((1-\alpha) w_W(z+y)-\alpha) y.\end{equation}
Applying the functional $\varphi$ to both sides of this equation we get
\begin{equation}\label{eq:second}x-\varphi(x)y=\alpha(z+y)\quad \&\quad \varphi(x)=(1-\alpha) w_W(z+y)-\alpha.\end{equation}
 More precisely, applying $\varphi$ to \eqref{eq:first} we get the second equation and plugging it into \eqref{eq:first} we get the first equation.
If we plug $z+y=\frac1\alpha(x-\varphi(x)y)$ to the second equation,
we get the quadratic equation
$$\alpha^2+\alpha(\varphi(x)+w_W(x-\varphi(x)y))-w_W(x-\varphi(x)y)=0.$$
If $w_W(x-\varphi(x)y)>0$, then this equation has one positive root and one negative root. Denote the positive root by $\alpha(x)$.
If $w_W(x-\varphi(x)y)=0$ and $\varphi(x)<0$, then the equation has one root equal to zero and the other one $\alpha(x)=-\varphi(x)>0$. If 
$w_W(x-\varphi(x)y)=0$ and $\varphi(x)\ge 0$, the equation has no positive root.

Since the conditions $w_W(x-\varphi(x)y)=0$ and $\varphi(x)\ge 0$
hold if and only if $x\in\cs V+[0,+\infty)y$ we get that the $\alpha$ in \eqref{eq:zero} is always unique and it follows from the first equation in \eqref{eq:second} that the corresponding $z$ is also unique. We denote it by $z(x)$. This finishes the proof of Step 1.
Moreover, the above calculation shows that the mappings $x\mapsto z(x)$ and $x\mapsto \alpha(x)$ are continuous on $X\setminus(\cs V+[0,+\infty)y)$.

\medskip

\noindent{\sc Step 2.} $H$ is a homeomorphism of $X\setminus(\cs V+[0,+\infty)y)$ onto itself.

\smallskip

It is clear that $H$ is a bijection of $X\setminus(\cs V+[0,+\infty)y)$ onto itself. So, it is enough to show that $H$ and $H^{-1}$ are continuous on $X\setminus(\cs V+[0,+\infty)y)$. This can be done using Lemma~\ref{l:Mink}
and continuity of $z(x)$ and $\alpha(x)$. 

More precisely, let us define $F$, a function of four real variables, on the set $$M=\{(\alpha,\beta,\gamma,\delta)\in(0,+\infty)^4: \gamma>\beta\ \&\ \delta>\beta\}$$  by the formula
$$F(\alpha,\beta,\gamma,\delta)=\begin{cases} \alpha & 0<\alpha\le\beta,\\
\beta+ \frac{\delta-\beta}{\gamma-\beta}(\alpha-\beta) &  \beta\le\alpha\le\gamma,\\ \alpha+\delta-\gamma & \gamma\le\alpha.
\end{cases}$$
This function is continuous on its domain, since all the three formulas are continuous, their domains are relatively closed and the formulas agree on the intersections of their domains.

Further, 
$$\begin{aligned} H(x) = &c(z(x))+ F(\alpha(x),\tfrac1{w_{V-c(z(x))}(z(x)-c(z(x)))},
\tfrac1{w_{U-c(z(x))}(z(x)-c(z(x)))},1)\,\cdot\\&
\qquad\qquad\qquad\qquad\qquad\qquad\qquad\qquad\qquad\quad\qquad\qquad\cdot	(z(x)-c(z(x)),\\
H^{-1}(x) = &c(z(x))+ F(\alpha(x),\tfrac1{w_{V-c(z(x))}(z(x)-c(z(x)))},
1,\tfrac1{w_{U-c(z(x))}(z(x)-c(z(x)))})\,\cdot\\&
\qquad\qquad\qquad\qquad\qquad\qquad\qquad\qquad\qquad\quad\qquad\qquad\cdot(z(x)-c(z(x)),
\end{aligned}$$
so both $H$ and $H^{-1}$ are continuous.

\medskip

\noindent{\sc Step 3.}  $H$ is a homeomorphism of $X$ onto itself.

\smallskip

Since $\cs V+[0,+\infty)y\subset\cc V\subset \Int V$ and $H$ is the identity on $V\setminus(\cs V+[0,+\infty)y)$, the global continuity of $H$ and $H^{-1}$ follows.

\begin{remark} Lemma 1.3 in \cite{BeKl} we have mentioned above is more general. It deals with homeomorphisms of triples, not pairs. To the set $V$ from \cite{BeKl} corresponds our set $U$, the sets $U$ and $P$ from \cite{BeKl} in our case coincide both with $V$.
The `tedious but not difficult' part skipped in \cite{BeKl} corresponds to our Steps 1 and 2. It is clear that the computation is not difficult, but especially Step 1 probably cannot be seen without a computation.
\end{remark} 

\subsection{The incorrect proof in \cite{BePe}}\label{s:BP}

On page 112 of the quoted book the authors suggest the formulas $c(z)=(w_U(z+y)-1)y$ and $v(z)=\frac12(u(z)+c(z))$. Further, $h_z$ is defined as the identity on $(c(z),v(z)]$ and on the halfline $v(z)+[0,+\infty)(z-c(z))$ as an affine mapping fixing $v(z)$ and taking $u(z)$ to $z$. 

We shall see that these
formulas do not provide a homeomorphism. The problem is that if $z+y\in\cc U$, we get $c(z)=-y$. In such a case $u(z)$ should be defined to be $z$ and already the mapping $z\mapsto u(z)$ may fail to be continuous. 

Let us describe a counterexample. Set $X=\er^3$ and 
$$U=\{(x_1,x_2,x_3)\in\er^3: x_1\ge (x_2)^+-1\ \&\ x_1\ge (x_3)^+-1\}.$$
Then $0\in\Int U$ and one can choose $y=(1,0,0)$ and $$Z=\{(x_1,x_2,x_3)\in\er^3:x_1=-1\}.$$

Let $x=(-1,-1,0)$. Then $x\in Z$, $w_U(x+y)=0$, hence $c(x)=-y$, $u(x)=x$ and hence $H(x)=x$.

Further, for any $n\in\en$ let $x_n=(-1,-1,\frac1n)$. Then $x_n\in Z$, $w_U(x_n+y)=\frac1n$, hence $c(x_n)=(-1+\frac1n,0,0)$. Further, 
$u(x_n)=(-1+\frac1{2n},-\frac12,\frac1{2n})$ as this point belongs is  the intersection of the boundary of $U$ with the segment $[c(x_n),x_n]$. Hence $v(x_n)=(-1+\frac3{4n},-\frac14,\frac1{4n})$ and 
$H(x_n)= v(x_n)+3(x_n-v(x_n))=(-1-\frac3{2n},-\frac52,\frac5{2n})$.

Since $x_n\to x$ and $H(x_n)\to(-1,-\frac52,0)\ne H(x)$, $H$ is not continuous.

\subsection{Correction of the proof -- version 1}\label{s:C1}

In this section we present a possible correction of the proof from \cite{BePe}.
We change the formula for $c(z)$ with preserving the remaining assumptions. 
Let us set $c(z)=(\sqrt{w_U(z+y)}-1)y$.

In this case the equality $c(z)=-y$ remains possible, but the square root changes certain order of convergence and makes the respective mappings continuous. This version of the proof is the most complicated one but we find it interesting.
So, let us give a proof.

\medskip

\noindent{\sc Step 1.} Set $Z'=\{z\in Z: w_U(z+y)>0\}$. Then the halflines $c(z)+(0,+\infty)(z-c(z))$, $z\in Z'$, are disjoint and cover the set $\{x\in X: w_U(x-\varphi(x)y)>0\}$.

\smallskip

Let $x\in X$. We will find out under which conditions there is $z\in Z'$ and $\alpha>0$ such that $$x=c(z)+\alpha(z-c(z)).$$
This equation is equivalent to
$$(x-\varphi(x) y)+\varphi(x) y 
= \alpha(z+y) + ((1-\alpha)(\sqrt{w_U(z+y)}-1)-\alpha) y.$$
So, by applying $\varphi$ to both sides we get (similarly as in ``\eqref{eq:first}$\implies$\eqref{eq:second}'' above)
\begin{equation}\label{eq:fourth}x-\varphi(x)y=\alpha(z+y) \quad \&\quad \varphi(x)= (1-\alpha)(\sqrt{w_U(z+y)}-1)-\alpha.\end{equation}
From the first equation it follows that $w_U(x-\varphi(x)y)>0$
if we want $z\in Z'$. Further, if we isolate $z+y$ from the first equation 
and plug the result into the second one, we get
$$\alpha\sqrt{w_U(x-\varphi(x)y)}+ \sqrt{\alpha}(\varphi(x)+1)-\sqrt{w_U(x-\varphi(x)y)}=0.$$
This is a quadratic equation for $\sqrt\alpha$ with a unique positive root $\alpha=\alpha(x)$. Hence, by the first equation in \eqref{eq:fourth}, there is a unique $z=z(x)$.

This completes the proof of Step 1. Moreover, the computation shows that the mappings $x\mapsto \alpha(x)$ and $x\mapsto z(x)$ are continuous on
$\{x\in X: w_U(x-\varphi(x)y)>0\}$.

\medskip

\noindent{\sc Step 2.} $H$ is a homeomorphism of $\{x\in X: w_U(x-\varphi(x)y)>0\}$ onto itself.

\smallskip

It is clear that $H$ is a bijection of the respective set onto itself. It remains to show that $H$ and $H^{-1}$ are continuous.

Let us define two functions of two real variables on the set $\er\times(0,2)$
by the formulas
$$\begin{aligned} G_1(\alpha,\beta)&=\begin{cases} \alpha & \alpha\le \frac\beta2,\\
\frac\beta2+\frac{2-\beta}{\beta}(\alpha-\frac\beta2) & \alpha\ge\frac\beta2,\end{cases}\\
G_2(\alpha,\beta)&=\begin{cases} \alpha & \alpha\le \frac\beta2,\\
\frac\beta2+\frac{\beta}{2-\beta}(\alpha-\frac\beta2) & \alpha\ge\frac\beta2.\end{cases}\end{aligned}$$
These functions are clearly continuous (the individual formulas are continuous, coincide on the intersection of the domains and the domains are relatively closed). Further, for $x\in\{x\in X: w_U(x-\varphi(x)y)>0\}$ we have
$$\begin{aligned}
H(x)&=c(z(x))+G_1(\alpha(x),\frac1{w_{U-c(z(x))}(z(x)-c(z(x)))})(z(x)-c(z(x))),\\ 
H^{-1}(x)&=c(z(x))+G_2(\alpha(x),\frac1{w_{U-c(z(x))}(z(x)-c(z(x)))})(z(x)-c(z(x))).\end{aligned}$$
It follows from Lemma~\ref{l:Mink} using the continuity of mappings $x\mapsto\alpha(x)$, $x\mapsto z(x)$ and $z\mapsto c(z)$ and the fact that $c(z(x))\in \Int U$ in this case that $H$ and $H^{-1}$ are continuous.
 
\medskip

\noindent{\sc Step 3:} $H$ is a homeomorphism of $X$ onto itself.

\smallskip

On the set $\{x\in X: w_U(x-\varphi(x)y)=0\}$ the mapping $H$ is defined to be identity. Since this set is closed, it is enough to show that whenever $x_\tau$ is a net in $\{x\in X: w_U(x-\varphi(x) y)>0\}$ such that $x_\tau\to x$ with $w_U(x-\varphi(x)y)=0$, then $H(x_\tau)\to x$ and $H^{-1}(x_\tau)\to x$.

So, let $(x_\tau)$ be such a net. Let us decompose the index set into two parts:
$$\begin{aligned}\Lambda_1&=\{\tau: w_{U-c(z(x_\tau))}(z(x_\tau)-c(z(x_\tau)))\le\frac1{2\alpha(x_\tau)}\},\\
\Lambda_2&=\{\tau: w_{U-c(z(x_\tau))}(z(x_\tau)-c(z(x_\tau)))>\frac1{2\alpha(x_\tau)}\}.\end{aligned}$$
For $\tau\in\Lambda_1$ we have $H(x_\tau)=H^{-1}(x_\tau)=x_\tau$, so it remains to show that the limit along $\Lambda_2$ is also $x$, provided $\Lambda_2$ is cofinal. Without loss of generality we may assume that $\Lambda_1=\emptyset$, i.e.
$$w_{U-c(z(x_\tau))}(z(x_\tau)-c(z(x_\tau)))>\frac1{2\alpha(x_\tau)}\mbox{\qquad for all }\tau.$$

Let us further compute the limit of $c(z(x_\tau))$. We have
$$c(z(x_\tau))=(\sqrt{w_U(z(x_\tau)+y)}-1)y =
(\sqrt{\frac{w_U(x_\tau-\varphi(x_\tau)y)}{\alpha(x_\tau)}}-1)y.$$
Since
$$\begin{aligned}\lim_\tau \sqrt{\frac{w_U(x_\tau-\varphi(x_\tau)y)}{\alpha(x_\tau)}}&=
\lim_\tau  \frac{{\sqrt{w_U(x_\tau-\varphi(x_\tau)y)}}\cdot 2\sqrt{w_U(x_\tau-\varphi(x_\tau)y)}}{-(\varphi(x_\tau)+1)+\sqrt{(\varphi(x_\tau) +1)^2+4w_U(x_\tau-\varphi(x_\tau)y)}}
\\&=\lim_\tau \frac12((\varphi(x_\tau)+1)+\sqrt{(\varphi(x_\tau)
+1)^2+4w_U(x_\tau-\varphi(x_\tau)y)})
\\ &= \frac12((\varphi(x)+1)+\sqrt{(\varphi(x)
+1)^2+4w_U(x-\varphi(x)y)})
\\& =(\varphi(x)+1)^+,
\end{aligned}$$
we get $c(z(x_\tau))\to ((\varphi(x)+1)^+-1)y=\max(\varphi(x),-1)y$.

If $\varphi(x)>-1$, then $c(z(x_\tau))\to\varphi(x)y\in\Int U$, hence by Lemma~\ref{l:Mink} we get 
$$w_{U-c(z(x_\tau))}(x_\tau-c(z(x_\tau)))\to w_{U-\varphi(x)y}(x-\varphi(x)y)=0$$
and hence $w_{U-c(z(x_\tau))}(x_\tau-c(z(x_\tau)))<\frac12$ for large $\tau$. It means that for large $\tau$ we have $\tau\in\Lambda_1$, a contradiction.

Thus $\varphi(x)\le -1$. Then $c(z(x_\tau))\to -y$.   
We will show that $$w_{U-c(z(x_\tau))}(z(x_\tau)-c(z_(x_\tau)))\to 1.$$ Suppose it is not the case. Since 
$w_{U-c(z)}(z-c(z))\ge 1$ for each $z\in Z'$, up to passing to a subnet we may assume that there is some $d>1$ such that $$w_{U-c(z(x_\tau))}(z(x_\tau)-c(z(x_\tau)))>d\mbox{ for each }\tau.$$ It means that
$z(x_\tau)-c(z(x_\tau))\notin d(U-c(z(x_\tau)))$, hence
$$\frac1dz(x_\tau) + (1-\frac1d) c(z(x_\tau))\notin U\mbox{ for each }\tau.$$
So,
\begin{multline*}\frac{w_U(z(x_\tau)+y)}d \cdot \frac{z(x_\tau)+y}{w_U(z(x_\tau)+y)}
+(1-\frac{w_U(z(x_\tau)+y)}d)\cdot(-y) \\
+((1-\frac1d)\sqrt{w_U(z(x_\tau)+y)} - \frac{w_U(z(x_\tau)+y)}d)y\notin U.\end{multline*}
Since $w_U(z(x_\tau)+y)\to 0$ the sum of the first two terms is for $\tau$ large enough a convex combination of $\frac{z(x_\tau)+y}{w_U(z(x_\tau)+y)}$ and $-y$, hence it belongs to $U$. Further, the coefficient at the last term is positive for $\tau$ large enough (this is the place where the choice of the square root is essential) which yields a contradiction as $y\in\cc U$.
Thus indeed $w_{U-c(z(x_\tau))}(z(x_\tau)-c(z_(x_\tau)))\to 1$.

Now we are ready to conclude.  
To shorten the notation, set $\alpha_\tau=\alpha(x_\tau)$ and $w_\tau= 	w_{U-c(z(x_\tau))}(z(x_\tau)-c(z(x_\tau)))$, Since $\alpha_\tau>\frac1{2w_\tau}$, we have

$$\begin{aligned}
H(x_\tau)&=c(z(x_\tau))+G_1(\alpha_\tau,\frac1{w_\tau})(z(x_\tau)-c(z(x)))\\ 
&=c(z(x_\tau))+(\alpha_\tau(2w_\tau-1)-1+\frac1{w_\tau})(z(x_\tau)-c(z(x_\tau)))
\\&=c(z(x_\tau))+(2w_\tau-1-\frac1{\alpha_\tau}+\frac1{\alpha_\tau w_\tau})(x_\tau-c(z(x_\tau)))\\
&=2c(z(x_\tau))(1-w_\tau)+x_\tau(2w_\tau-1)+\frac{1-w_\tau}{\alpha_\tau w_\tau}(x_\tau-c(z(x_\tau)))\to x\end{aligned}$$
since $x_\tau\to x$, $c(z(x_\tau))\to-y$, $w_\tau\to 1$ and $\alpha_\tau w_\tau>\frac12$.

Similarly,
$$\begin{aligned}
H^{-1}(x_\tau)&=c(z(x_\tau))+G_2(\alpha_\tau,\frac1{w_\tau})(z(x_\tau)-c(z(x)))\\ 
&=c(z(x_\tau))+(\frac{\alpha_\tau}{2w_\tau-1}+\frac{w_\tau-1}{w_\tau(2w_\tau-1)})(z(x_\tau)-c(z(x_\tau)))
\\&=c(z(x_\tau))+(\frac{1}{2w_\tau-1}+\frac{w_\tau-1}{\alpha_\tau w_\tau(2w_\tau-1)})(x_\tau-c(z(x_\tau)))\to x.\end{aligned}$$

This completes the proof.

\subsection{Correction of the proof -- version 2}\label{s:C2}

Another possibility how to correct the proof is to use the formula $c(z)=(w_U(z+y)+1)y$ for $z\in Z$.
In this case the problem appearing in the original version and in the first correction disappears, since $c(z)\in \Int U$ for all $z\in Z$. Let us show that this modification works.

\medskip

\noindent{\sc Step 1.} The halflines $c(z)+(0,+\infty)(z-c(z))$, $z\in Z$, are pairwise disjoint and their union is the set $X\setminus ((\Ker\varphi\cap\cc U)+[1,+\infty)y)$.

\smallskip

Let $x\in X$. Let us find out under which conditions there are $z\in Z$ and $\alpha>0$ such that
$$x= c(z)+\alpha(z-c(z)).$$
This equation is equivalent to
$$(x-\varphi(x)y)+\varphi(x)y=\alpha(z+y)+((1-\alpha)(w_U(z+y)+1)-\alpha)y.$$
By applying $\varphi$ to both sides we see that the above equation  is equivalent to (similarly as in ``\eqref{eq:first}$\implies$\eqref{eq:second}'' above)
\begin{equation}\label{eq:fifth}x-\varphi(x)y=\alpha(z+y)\quad\&\quad\varphi(x)= (1-\alpha)(w_U(z+y)+1)-\alpha.\end{equation}
From the first equation isolate $z+y$ and plug it to the second equation. We get thus
a quadratic equation for $\alpha$:
$$2\alpha^2+\alpha(\varphi(x)-1+w_U(x-\varphi(x)y))-w_U(x-\varphi(x)y)=0.$$
If  $w_U(x-\varphi(x)y)>0$, there is a unique positive root $\alpha=\alpha(x)$.
If $w_U(x-\varphi(x)y)=0$ and $\varphi(x)<1$, there is a unique positive root $\alpha(x)=(1-\varphi(x))/2$.
If $w_U(x-\varphi(x)y)=0$ and $\varphi(x)\ge1$ (i.e., if $x\in(\Ker\varphi\cap\cc U)+[1,+\infty)y$), then there is no positive root. This shows there is a unique $\alpha=\alpha(x)$ and it follows from the first equation in \eqref{eq:fifth} that there is also a unique $z=z(x)$. This completes the proof of Step 1. Moreover, the computation shows that the mappings $x\mapsto \alpha(x)$ and $x\mapsto z(x)$ are continuous on $X\setminus ((\Ker\varphi\cap\cc U)+[1,+\infty)y)$.

\medskip

\noindent{\sc Step 2.} $H$ is a homeomorphism of $X\setminus ((\Ker\varphi\cap\cc U)+[1,+\infty)y)$ onto itself.

\smallskip

It is clear that $H$ is a bijection of the mentioned set onto itself.
Further, the formulas for $H$ and $H^{-1}$ are the same as in the previous case. Of course, $z(x)$, $\alpha(x)$ and $c(z(x))$ are given by different formulas, but since these mappings are continuous, we get that $H$ and $H^{-1}$ are continuous.

\medskip

\noindent{\sc Step 3.} $H$ is a homeomorphism of $X$ onto itself.

\smallskip

Since $H$ is defined as the identity on the set $(\Ker\varphi\cap\cc U)+[1,+\infty)y$ and this set is closed in $X$, it is enough to show the following: Let $(x_\tau)$ be a net in $X\setminus ((\Ker\varphi\cap\cc U)+[1,+\infty)y)$ converging to some $x\in (\Ker\varphi\cap\cc U)+[1,+\infty)y$.
Then $H(x_\tau)\to x$ and $H^{-1}(x_\tau)\to x$. So, let us have such a net.

We have $x_\tau=c(z(x_\tau))+\alpha(x_\tau)(z(x_\tau)-c(z(x_\tau)))$. We will show that for $\tau$ large enough $\alpha(x_\tau)\le\frac1{2w_{U-c(z(x_\tau))}(z(x_\tau)-c(z(x_\tau)))}$. Then the proof will be completed as it will follow that for $\tau$ large enough $H(x_\tau)=H^{-1}(x_\tau)=x_\tau$.

The desired inequality is equivalent to $w_{U-c(z(x_\tau))}(z(x_\tau)-c(z(x_\tau)))\le\frac1{2\alpha(x_\tau)}$, i.e.
$c(z(x_\tau))+2\alpha(x_\tau)(z(x_\tau)-c(z(x_\tau)))\in U$, equivalently
$$c(z(x_\tau))+2(x_\tau-c(z(x_\tau)))\in U.$$

Let us analyze the limit behaviour of the left-hand side. Set $t_\tau=\varphi(x_\tau)$ and $a_\tau=x_\tau-t_\tau y$. Then $t_\tau\to\varphi(x)$ and $a_\tau\to x-\varphi(x)y$, hence $w_U(a_\tau)\to 0$. 
We have
$$c(z(x_\tau))=(w_U(z(x_\tau)+y)+1)y=(\frac{w_U(a_\tau)}{\alpha(x_\tau)}+1)y.$$
If $w_U(a_\tau)=0$, then $c(z(x_\tau))=y$. Moreover, $t_\tau<1$, hence $\varphi(x)\le 1$, so necessarily 
$\varphi(x)=1$. It follows that $c(z(x_\tau))=\varphi(x)y$.

If $w_U(a_\tau)>0$, then $\alpha(x_\tau)$ is the positive root of the quadratic equation from Step 1, 
so 
$$\frac{w_U(a_\tau)}{\alpha(x_\tau)}=\frac12(\sqrt{(t_\tau-1 +w_U(a_\tau))^2+8w_U(a_\tau)}+t_\tau-1+w_U(a_\tau))\to \varphi(x)-1.$$
It follows that $c(z(x_\tau))\to\varphi(x)y$, hence
$$c(z(x_\tau))+2(x_\tau-c(z(x_\tau)))\to \varphi(x)y+2(x-\varphi(x)y).$$
Since $x-\varphi(x)y\in\cc U$, $y\in\cc U$ and $\varphi(x)\ge 1$, we get 
 $\varphi(x)y+2(x-\varphi(x)y)\in\cc U\subset\Int U$. Hence 
$c(z(x_\tau))+2(x_\tau-c(z(x_\tau)))\in U$ for $\tau$ large enough and the proof is completed.

\begin{remark} It took us some time to discover that the proof in \cite{BePe} is incorrect. As remarked above, the error is already in the second step, since the assignment $z\mapsto u(z)$ fails to be continuous. The correction from Section~\ref{s:C1} is quite complicated but we find it interesting since it uses
some balance of asymptotic behaviour. The correction from Section~\ref{s:C2} is much simpler and now, a posteriori, we are convinced that this is the formula the authors had in mind. But it is still more complicated than the original proof, the main difference is in Step 3. While in the original version Step 3 is trivial, in the method described in Section~\ref{s:C2} Step 3 requires some nontrivial computation.
At least we do not see how to prove it without any computation like in the original version.
\end{remark}

\section{Topological version of Dobrowolski's proof}\label{s:D} 

The approach of \cite{Dob} is a bit different, it focuses on smooth bodies in Banach spaces and refers to the implicit function theorem. As remarked above, the proof is extremely consise and missing computations (checking the assumptions of the implicit function theorem) is nontrivial -- it would be much longer than the proof itself. 
In this section we give a modification of the proof from \cite{Dob} which works simultaneously in the topological and the smooth cases. Our version is moreover simplified and more elementary. In particular, it uses a simpler version of the implicit function theorem (not only its proof is simpler, but the assumptions are easier to check) and the form of our formula is simpler (although the mapping is the same) since we use Minkowski functional related to only one convex body.  We will give the proof in the topological case and then comment why it works also in the smooth case. 

Firstly, let us choose two auxiliary $C^\infty$ functions $\lambda$ and $\gamma$ defined on $\er$ with the following properties:
\begin{itemize}
	\item $\lambda$ is non-decreasing, $\lambda=0$ on $(-\infty,\frac12]$, $\lambda=1$ on $[1,+\infty)$.
	\item $\gamma=0$ on $(-\infty,\frac12]$, $\lim_{t\to\infty}\gamma(t)=+\infty$ and $0\le\gamma'(t)<\frac1t(\gamma(t)+1)$ for $t>0$.
\end{itemize}
The existence of $\lambda$ is a well-known fact. The existence of $\gamma$ is not obvious and in \cite{Dob} it is just postulated. One can take, for example, $\gamma(t)=\delta\lambda(t)\ln(t+1)$
for $t>-1$, where $\delta>0$ is a small enough number and complete this function by zero on the rest of $\er$. 

In the proof we will need the following version of the implicit function theorem.

\begin{thm}\label{IFT} Let $X$ be a topological space, $\Omega\subset X\times \er$ an open set,
$F=F(x,t):\Omega\to \er$ a function and $(x_0,t_0)\in \Omega$. Suppose that the following assumptions
are satisfied.
\begin{itemize}
	\item $F$ and $\frac{\partial F}{\partial t}$ are continuous on $\Omega$.
	\item $F(x_0,t_0)=0$.
	\item  $\frac{\partial F}{\partial t}(x_0,t_0)\ne0$.
\end{itemize}
Then there is $G$, a neighborhood of $x_0$ in $X$, and $H$, a neighborhood of $t_0$ in $\er$ and a continuous function $f:G\to H$ such that $G\times H\subset\Omega$ and for $(x,t)\in G\times H$ one has $t=f(x)$ if and only if $F(x,t)=0$.
\end{thm}

This theorem follows from a more general \cite[Chapter III, Section 8, Theorem 25]{Schw} (which deals with a normed space in place of $\er$). However, our version is much simpler and can be proved by the same way as the easiest version
for $C^1$ functions from $\er^2$ to $\er$.

Let us now start the construction itself.  For $z\in Z$ let $c(z)=\gamma(w_U(y+z))y$.

\medskip

\noindent{\sc Step 1:} The mapping $\Phi: (\alpha,z)\mapsto c(z)+\alpha(z-c(z))$ is a homeomorphism of $(0,+\infty)\times Z$ onto $X\setminus ((\cc U\cap \Ker\varphi)+[0,+\infty)y)$.

\smallskip

Note that in \cite{Dob} there is a small error, where instead of $\cc U\cap \Ker\varphi$ the author writes $\{0\}$. We proceed in the same way as above. Fix $x\in X$ and try to find $\alpha>0$ and $z\in Z$ such that $\Phi(\alpha,z)=x$. This equation is equivalent to
$$\alpha(z+y)+((1-\alpha)\gamma(w_U(z+y))-\alpha)y=(x-\varphi(x)y)+\varphi(x)y,$$
hence by applying $\varphi$ to both sides we get (similarly as in ``\eqref{eq:first}$\implies$\eqref{eq:second}'' above)
$$\alpha(z+y)=x-\varphi(x)y\quad\&\quad(1-\alpha)\gamma(w_U(z+y))-\alpha=\varphi(x).$$
If we isolate $z+y$ from the first equation and plug it in the second one,
we get
$$(1-\alpha)\gamma\left(\frac1\alpha w_U(x-\varphi(x)y)\right)-\alpha-\varphi(x)=0.$$
Denote the left-hand side by $F(x,\alpha)$. It is clear that $F$ is defined and continuous on $X\times(0,+\infty)$ and, moreover,
$$\frac{\partial F}{\partial \alpha}(x,\alpha)=
-\gamma\left(\frac1\alpha w_U(x-\varphi(x)y)\right)-\frac{1-\alpha}{\alpha^2}w_U(x-\varphi(x)y)
\gamma'\left(\frac1\alpha w_U(x-\varphi(x)y)\right)-1,$$
which is also continuous on $X\times(0,+\infty)$.

Further, $\frac{\partial F}{\partial \alpha}(x,\alpha)<0$ for $(x,\alpha)\in X\times(0,+\infty)$.
Indeed, if $w_U(x-\varphi(x)y)=0$, then $\frac{\partial F}{\partial \alpha}(x,\alpha)=-1$. If $w_U(x-\varphi(x)y)>0$
and $\alpha\le 1$, then $\frac{\partial F}{\partial \alpha}(x,\alpha)\le-1$. Finally, if $w_U(x-\varphi(x)y)>0$
and $\alpha>1$, then by the properties of $\gamma$ we get
$$\frac{\partial F}{\partial \alpha}(x,\alpha)
<-\gamma\left(\frac1\alpha w_U(x-\varphi(x)y)\right)+\frac{\alpha-1}{\alpha^2}
\left(\gamma\left(\frac1\alpha w_U(x-\varphi(x)y)\right)+1\right)-1\le0.$$

Let us continue with describing the range of $\Phi$. Fix $x\in X$. There are two possibilites:

\smallskip

Case 1: $w_U(x-\varphi(x)y)=0$. Then $F(x,\alpha)=-\alpha-\varphi(x)$. If $\varphi(x)<0$, there is a unique positive root
$\alpha=-\varphi(x)$. If $\varphi(x)\ge0$, there is no positive root.

\smallskip

Case 2: $w_U(x-\varphi(x)y)>0$. Then $\lim_{\alpha\to 0+}F(x,\alpha)=+\infty$ (as $\gamma$ has at $+\infty$ the limit $+\infty$)
and $\lim_{\alpha\to+\infty}F(x,\alpha)=-\infty$ (as $\gamma$ vanishes at a neighborhood of zero). Further, since $\alpha\mapsto F(x,\alpha)$ is continuous and strictly decreasing on $(0,+\infty)$, there is a unique root.

\smallskip

It follows that $\Phi$ is one-to-one and its range is $X\setminus ((\cc U\cap \Ker\varphi)+[0,+\infty)y)$.
It is clear that $\Phi$ is continuous. By Theorem~\ref{IFT} we get that the second coordinate of the inverse is continuous,
the continuity of the first coordinate then follows, hence $\Phi^{-1}$ is continuous.

\medskip

\noindent {\sc Step 2.} The mapping $\Psi$ defined by the formula 
$$\Psi(\alpha,z)=(\alpha\lambda(\alpha w_{U-c(z)}(z-c(z))) (w_{U-c(z)}(z-c(z))-1)
 + \alpha,z)$$
is a homeomorphism of $(0,+\infty)\times Z$ onto itself. 

\smallskip

Since $w_{U-c(z)}(z-c(z))\ge 1$ whenever $z\in Z$, $\Psi$ maps $(0,+\infty)\times Z$ into itself. $\Psi$ is clearly continuous. To show that $\Psi$ is a bijection and the inverse is continuous, let us investigate the first coordinate, i.e., the mapping $$\theta(\alpha,z)=\alpha\lambda(\alpha w_{U-c(z)}(z-c(z))) (w_{U-c(z)}(z-c(z))-1)  + \alpha.$$
We have
\begin{multline*}\frac{\partial}{\partial\alpha}\theta(\alpha,z)
= (	\lambda(\alpha w_{U-c(z)}(z-c(z)))\\+\alpha\lambda'(\alpha w_{U-c(z)}(z-c(z)))(w_{U-c(z)}(z-c(z))))(w_{U-c(z)}(z-c(z))-1)+1.\end{multline*}
This partial derivative is continuous and strictly positive on $(0,+\infty)\times Z$.
Moreover, for any $z\in Z$ we have
$$\lim_{\alpha\to0+}\theta(\alpha,z)=0 \mbox{ and }\lim_{\alpha\to+\infty}\theta(\alpha,z)=+\infty,$$
hence $\Psi$ is a bijection. Moreover, the continuity of $\Psi^{-1}$ follows from 
Theorem~\ref{IFT}. This completes the proof of Step 2.
 
\medskip
 
\noindent{\sc Step 3.} The mapping $H=\Phi\circ\Psi\circ\Phi^{-1}$ is a homeomorphism of $X\setminus ((\cc U\cap \Ker\varphi)+[0,+\infty)y)$ onto itself. Moreover, 
it maps each halfline $c(z)+\er^+(z-c(z))$ onto itself in an increasing manner such that the segment $(c(z),u(z)]$ is mapped onto $(c(z),z]$.

\smallskip

Indeed, $H$ is a homeomorphism as a composition of homeomorphisms. Further, from the construction it is clear that it preserves the mentioned halflines in an increasing manner. The last thing to show is that $H(u(z))=z$. To show this notice first
that $\Phi^{-1}(u(z))=(\frac1{w_{U-c(z)}(z-c(z))},z)$, hence
$\Psi(\Phi^{-1}(u(z))=(1,z)$, thus $H(u(z))=z$.

\medskip

\noindent{\sc Step 4.} If we extend $H$ by identity on $(\cc U\cap \Ker\varphi)+[0,+\infty)y$ we get a homeomorphism of $X$ onto itself with the required properties.

\smallskip

Since $(\cc U\cap \Ker\varphi)+[0,+\infty)y$ is closed, it is enough to check that $H$ and $H^{-1}$ are continuous at points of this set. So, let us fix $x$ in this set
and a net $x_\tau$ in the complement converging to $x$. Let $(\alpha_\tau,z_\tau)=\Phi^{-1}(x_\tau)$. Let us first show that $\alpha_\tau\to 0$.
Suppose not. Then, up to passing to a subnet, we may assume that $\alpha_\tau\to\alpha\in(0,+\infty]$.
 We have
$$\varphi(x)=\lim_\tau\varphi(x_\tau)=\lim_{\tau}(1-\alpha_\tau)\gamma\left(\frac1{\alpha_\tau} w_U(x_\tau-\varphi(x_\tau)y)\right)-\alpha_\tau.$$
If $\alpha=+\infty$, then the limit on the right-hand side is $-\infty$ (since $\gamma$ is zero on a neighborhood of zero), which is not possible. If $\alpha\in(0,+\infty)$, then the right-hand side goes to $-\alpha$
(since $w_U(x_\tau-\varphi(x_\tau)y)\to w_U(x-\varphi(x)y)=0$). Thus $\varphi(x)<0$, a contradiction.

So, we have proved that $\alpha_\tau\to 0$. Further,
$$c(z_\tau)=\gamma(w_U(z_\tau+y))y =\gamma\left(\frac1{\alpha_\tau} w_U(x_\tau-\varphi(x_\tau)y)\right)y
=\frac{\varphi(x_\tau)+\alpha_\tau}{1-\alpha_\tau}y\to\varphi(x)y.$$
Hence $w_{U-c(z_\tau)}(x_\tau-c(z_\tau))\to w_{U-\varphi(x)y}(x-\varphi(x)y)$ by Lemma~\ref{l:Mink}.
But the latter value is zero, since $w_{U}(x-\varphi(x)y)=0$ and $y\in\cc U$, so for each $t>0$ we have $t(x-\varphi(x)y)+\varphi(x)y\in U$.
So, for $\tau$ large enough we have
$\frac1\alpha w_{U-c(z_\tau)}(z_\tau-c(z_\tau))= w_{U-c(z_\tau)}(x_\tau-c(z_\tau))<\frac12,$
hence $\Psi(\alpha_\tau,z_\tau)=(\alpha_\tau,z_\tau)$.
Finally, for those $\tau$ we have
$H(x_\tau)=H^{-1}(x_\tau)=x_\tau$.

This completes the proof.

\begin{remark} In case $X$ is a Banach space a $U$ is a $C^p$-smooth convex body (where $p\in\en\cup\{\infty\}$), the homeomorphism $H$ constructed above is a $C^p$-diffeomorphism. Indeed, first remark that, if in Theorem~\ref{IFT} we moreover assume that $X$ is a Banach space and $F$ is $C^p$-smooth, then $f$ is also $C^p$-smooth. Further, in this case the function from Lemma~\ref{l:Mink} is $C^p$-smooth on the complement of its zero set by \cite[Lemma 1]{Dob}. (The proof of this lemma is omitted in \cite{Dob}, but it is an easy consequence of the definition.) Further, the function $F$ used in Step 1 is $C^p$-smooth on $X\times(0,+\infty)$ (at points where $w_U(x_0-\varphi(x_0)y)>0$ this is a composition of $C^p$-functions mentioned above; if  $w_U(x_0-\varphi(x_0)y)=0$, then $F(x,\alpha)=-\alpha-\varphi(x)$ on a neighborhood of $(x_0,\alpha_0)$). It follows that $\Phi$ is a $C^p$-diffeomorphism. Similarly we can see that the mapping $\Psi$ from Step 2 is a $C^p$-diffeomorphism. Finally, from the proof of Step 4 we see that
for each point from $(\cc U\cap \Ker\varphi)+[0,+\infty)y$ there is a neighborhood on which $H$ is the identity, so $H$ is a $C^p$-diffeomorphism.\end{remark}

\def\cprime{$'$}

\end{document}